\documentclass[11pt]{article}
\usepackage{amssymb}
\usepackage{amssymb}
\usepackage{amsfonts}
\usepackage{bm}
\usepackage{xcolor}
\usepackage{hyperref}
\hypersetup{
  colorlinks=false,
  linkbordercolor=red,linkcolor=green,pdfborderstyle={/S/U/W 1}
} 

\usepackage{tikz}
\usepackage{tikz-3dplot}
\usetikzlibrary{shapes.geometric, arrows,3d,positioning}

\usepackage{amssymb}
\usepackage[T1]{fontenc}
\usepackage{latexsym,amssymb,amsmath,amsfonts,amsthm}
\usepackage{graphics}
\usepackage{graphicx}
\usepackage{mathrsfs}
\usepackage{subfigure}
\usepackage{float}
\newcommand{\R}{{\mat R}}

\newcommand{\C}{{\mat C}}

\newcommand{\bT}{{\mathbf{ T}}}

\newcommand{\ben}{\begin{eqnarray*}}
\newcommand{\en}{\end{eqnarray}}
\newcommand{\enn}{\end{eqnarray*}}

\newcommand{\bF}{\mathbf{F}}

\renewcommand{\i}{{\rm i}}

\newcommand{\bgamma}{{\boldsymbol{\gamma}}}

\renewcommand{\d}{{\rm d}}

\newcommand{\mat}{\mathbb}

\renewcommand{\t}{{\bm t}}
\newcommand{\n}{{\bm n}}

\renewcommand{\epsilon}{\varepsilon}
\newtheorem{theorem}{Theorem}[section]
\newtheorem{lemma}[theorem]{Lemma}

\newtheorem{remark}[theorem]{Remark}
\newtheorem{example}[theorem]{Example}
\newtheorem{proposition}[theorem]{Proposition}
\newtheorem{assumption}[theorem]{Assumption}

\usepackage{color}

\begin{document}
\renewcommand{\theequation}{\arabic{section}.\arabic{equation}}

\begin{titlepage}
\title{\bf Directional Tangency and Anomalous Far-Field Decay for Schr{\"o}dinger Equations}
\author{Ruming Zhang\thanks{Institute for Mathematics, Technische Universit{\"a}t Berlin, Berlin, Germany; \texttt{ruming.zhang@tu-berlin.de}}
}
\end{titlepage}
\maketitle

\begin{abstract}
We identify a directional tangency mechanism responsible for anomalous far-field decay in Schrödinger equations. For constant-coefficient operators, we introduce a directional decomposition of the limiting resolvent based on the geometry of the Fermi surface and the observation direction. In two dimensions, we show that an analytic source produces a super-algebraically small tangency contribution, whereas a transverse non-analyticity located on the directional tangency set generates an algebraic contribution of order \(R^{-(2\alpha+1)}\). Shifting the singularity away from the tangency set restores super-algebraic decay, showing that geometric alignment is essential to the anomalous behavior. We further construct a bounded, decaying real-valued potential for which the same mechanism occurs in the first Born correction through the effective source \(Vu_0\). In a suitable long-range regime, the resulting tangency contribution decays more slowly than the conventional two-dimensional far-field term.
\end{abstract}


\section{Introduction}

In this paper, we consider the  Schrödinger equation in a $d$-dimensional ($d\geq 2$) space:
\begin{equation}
\label{eq:schrdg}
(H-\lambda I)u=f\quad\text{ in }\R^d
\end{equation}
for a positive valued frequency $\lambda$. The operator $H$ is defined as
\begin{equation*}
Hu:=-\nabla\cdot A\nabla u(x)+V(x) u(x)
\end{equation*}
where $A$ is a real-valued, positive definite matrix, $V$ is the  real-valued decaying potential,  $f$ is the  source term and $\lambda\in\R$ is the frequency. In particular, let $H_0$ denote the operator with $V=0$. 

 Furthermore, let the resolvent
\[
R(z):=(H-zI)^{-1}.
\] 
For $z=\lambda+\i\epsilon$ with $\epsilon>0$, $R(z)$ exists. Therefore, we define $R(\lambda)$ as the limit of $R(\lambda+\i\epsilon)$ for $\epsilon\rightarrow 0^+$ according to the limiting absorption principle. Therefore, we are interested in the asymptotic behaviour of $u$ with $V$ and $f$ with particular singularities. Moreover, $R_0(z)$ and $R_0(\lambda)$ are defined similarly for $V=0$.

\subsection{Background and previous work}

\bigskip
\noindent
{\em Scattering theory and limiting absorption principle.}
Scattering theory and the limiting absorption principle (LAP) provide a
standard framework for studying time-harmonic solutions at frequencies
belonging to the continuous spectrum. We refer to
\cite{Lax1967,Melrose1992,Vainberg1989} for general introductions to
scattering theory and related problems. For time-harmonic waves, the
LAP constructs limiting solutions through the resolvent
\[
R(\lambda\pm i\varepsilon)
=(H-\lambda\mp i\varepsilon)^{-1},
\qquad \varepsilon\downarrow0.
\]
The LAP and related asymptotic properties of solutions have been studied
extensively since the classical works of Agmon and Agmon--H{\"o}rmander
\cite{Agmon1975,Agmon1976}.

\bigskip
\noindent
{\em Limiting absorption principle for Schrödinger equations with
decaying potentials.}
For time-harmonic Schrödinger equations \eqref{eq:schrdg}, the limiting
absorption principle has been studied under various assumptions on the
potential, including decay and regularity conditions. Classical
spectral theory for Schrödinger operators can be found in
\cite{Reed1978}, while uniform Sobolev and resolvent estimates were
developed in \cite{Kenig1987}. Resolvent estimates and limiting
absorption principles in Lebesgue spaces were studied in
\cite{Gold2004,Iones2006}, and effective limiting absorption principles
were further developed in \cite{Rodni2015}. For general treatments of
the limiting absorption principle for Schrödinger operators, we refer
to \cite{Yafaev1992,Yafaev2010}.

\bigskip
\noindent
For the Schr{\"o}dinger equation \eqref{eq:schrdg}, the decay rate of
the potential is closely related to the far-field asymptotics and the
appropriate radiation condition. Short-range and long-range potentials
lead to different asymptotic regimes; see
\cite{Enss1978,Enss1979,Yafaev1992,Yafaev2010}. In the long-range
setting, non-negligible phase accumulation leads to modified
asymptotic behavior and, in particular, to modified radiation
conditions; see \cite{Isozaki1980,Gatel1999}. Radiation condition
bounds and limiting absorption principles under broad assumptions were
further developed by Ito and Skibsted \cite{Ito2020} and, more recently, by Larsen \cite{Larsen2024}.

\subsection{Main contributions}

This paper aims to identify a particular mechanism underlying the
far-field behavior of solutions to the Schr{\"o}dinger equation within
the limiting absorption framework. The work is motivated by my recent
preprint \cite{Zhang2026a} on higher-dimensional periodic elliptic
equations. In that work, I treat the observation direction as an
additional parameter in the spectral analysis and thereby reveal a
geometric structure that we call the directional tangency set. The
directional tangency set is a subset of the Fermi surface at which the
observation direction is tangent to the Fermi surface. It is not merely
a transition between the real and complex parts of the Fermi surface,
or between propagating and evanescent waves; it can also determine a
distinct component of the wave field, which may become a leading-order
contribution to the total field.

This observation raises a natural question:
\emph{does the directional tangency set continue to exist and play an
important role in a non-periodic setting?}

The simplest non-periodic setting is the constant-coefficient
Schr{\"o}dinger equation, for which the Fermi surface is an ellipsoid
determined by the positive definite matrix $A$. For a fixed observation
direction $n$, the directional tangency set is the subset of this
ellipsoid at which $n$ is tangent to the Fermi surface. This leads to the first question: can a singularity in the source term interact with the directional tangency set and generate a distinct tangency -region contribution
to the far field? In particular, what happens when the singularity is
located exactly on the directional tangency set, and what happens when
it is shifted away from it?

A second question concerns the effect of a non-constant potential.
If the potential $V$ is non-constant in $\mathbb{R}^d$, while the source
term $f$ is sufficiently simple, can the potential itself generate a
similar tangency -region contribution? This question is naturally related to
the first one through the Lippmann--Schwinger representation, in which
the potential enters through the effective source term $Vu$.

In this paper, we identify and rigorously analyze this mechanism in
two dimensions. For constant coefficients, we show that a non-smooth
singularity of the source term in the direction transverse to the
observation direction can produce a slowly decaying tangency
contribution when the singularity is located precisely on the
directional tangency set. In contrast, shifting the singularity away
from the directional tangency set causes the corresponding tangency
contribution to decay super-algebraically, and the anomalous decay
disappears. This shows that the singularity alone is not sufficient:
its geometric alignment with the directional tangency set is
essential.

We further construct a non-constant decaying potential for which the
same anomalous tangency mechanism appears already at the level of the
first Born approximation. More precisely, the potential is constructed
so that its interaction with the free outgoing solution produces an
effective source term with the prescribed transverse singularity.
Thus, the same tangency mechanism can be realized in a non-periodic
Schr{\"o}dinger setting and is not restricted to the periodic framework.
The construction also shows that the tangency -region contribution can
dominate the standard non-tangential far-field term for suitable
long-range potentials.

The construction also distinguishes the existence of the tangency mechanism from its asymptotic dominance. The tangency contribution persists as long as the effective source retains a transverse non-analyticity at the tangency point, whereas its decay rate depends quantitatively on the order of this non-analyticity. For the family constructed below, the same parameter that controls the spatial decay of the potential also controls the decay of the tangency contribution. In particular, the anomalously slow, dominant regime occurs only for a suitable range of the parameter corresponding to long-range potentials. Faster spatial decay of the potential weakens the tangency contribution, although it does not by itself remove the underlying tangency mechanism.

\section{Directional decomposition and tangency geometry}

\subsection{Settings and definitions}

We first consider the unperturbed case $V=0$ for a fixed $\lambda>0$, namely
\begin{equation}
\label{eq:schrdg_homo}
    -\nabla \cdot A\nabla u(x)-\lambda u(x)=f(x)
    \qquad \text{in } \mathbb{R}^d.
\end{equation}

We consider the limiting absorption resolvent by first introducing, for
$\varepsilon>0$, the damped problem
\begin{equation}
\label{eq:schrdg_homo_damp}
    -\nabla\cdot A\nabla u_\varepsilon(x)
    -(\lambda+\i\varepsilon)u_\varepsilon(x)
    =f(x)
    \qquad \text{in } \mathbb{R}^d.
\end{equation}
The source term $f$ satisfies the following assumption.

\begin{assumption}
\label{asp1}
Assume that the Fourier transform of $f$, denoted by $\hat{f}$, belongs to $L^1(\R^d)$.
\end{assumption}

In the Fourier variables,
\begin{equation*}
    \widehat{u}_\varepsilon(\xi)
    =
    \frac{\widehat{f}(\xi)}
    {\xi\cdot A\xi-\lambda-\i\varepsilon},
\end{equation*}
and hence
\begin{equation*}
    u_\varepsilon(x)
    =
    \frac{1}{(2\pi)^d}
    \int_{\mathbb{R}^d}
    \frac{e^{\i x\cdot\xi}\widehat{f}(\xi)}
    {\xi\cdot A\xi-\lambda-\i\varepsilon}
    \,d\xi
\end{equation*}
converges absolutely, since for  $\varepsilon>0$,
$
    \left|
    \xi\cdot A\xi-\lambda-\i\varepsilon
    \right|
    \geq \varepsilon.
$

The singular set of the limiting resolvent is the Fermi surface
\begin{equation*}
    \bF
    :=
    \left\{
        \xi\in\mathbb{R}^d:
        \xi\cdot A\xi=\lambda
    \right\}.
\end{equation*}
In the following, we fix an observation direction
$\n\in \mathbb{S}^{d-1}$ and decompose the Fourier variable into its components
parallel and orthogonal to $\n$. This leads to a family of one-dimensional
integrals in the longitudinal variable, which will be analyzed by contour
deformation.

The analyticity required for this contour deformation will be imposed
linewise, for each fixed transverse parameter. No regularity with respect
to the transverse parameter is required at this stage. Additional
regularity in the transverse parameter will be introduced later when the
resulting representations are integrated and their asymptotic behavior is
analyzed.

\subsection{Directional approach}

We fix an observation direction $\n\in S^{d-1}$. For $x\in
\mathbb{R}^d\setminus\{0\}$, we write
\[
    \n=\frac{x}{\lVert x\rVert},
    \qquad
    x=R\n,
    \qquad
    R=\lVert x\rVert>0.
\]
Thus, throughout the following analysis, $\n$ is fixed while
$R\to\infty$.

The active part of the Fermi surface in the direction $\n$ is defined by
\[
    \bF_+
    :=
    \left\{
        \xi\in\bF: \n\cdot A\xi>0
    \right\},
\]
while the directional tangency set is
\[
    \bT_n
    :=
    \left\{
        \xi\in\bF: \n\cdot A\xi=0
    \right\}.
\]
The complementary part
\[
    \bF_-
    :=
    \left\{
        \xi\in\bF: \n\cdot A\xi<0
    \right\}
\]
will also be used below. For the illustrations see Figure \ref{fig:sample}.

\begin{figure}[h]
    \centering
    \includegraphics[width=0.7\textwidth]{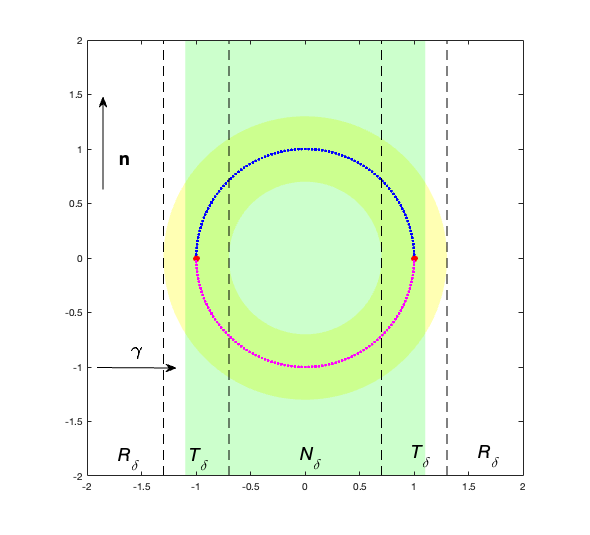}
    \caption{Structure of two-dimensional example. Blue dotted curve: $\bF_+$; pink dotted curve: $\bF_-$; red dots: $\bT_\n$; yellow annulus: $\mathcal{U}_\eta(\bF)$; green region: $\mathcal{C}_\delta$; the overlap between the yellow annulus and the green region: $\widetilde{\mathcal{U}}_\eta(\bF)$. }
    \label{fig:sample}
\end{figure}

Let
\[
    \{\t_1,\ldots,\t_{d-1}\}
\]
be an orthonormal basis of $\n^\perp$. We decompose
\[
    \xi=\bgamma \t+s\n,
    \qquad s\in\mathbb{R},
\]
where
\[
    \bgamma \t
    :=
    \gamma_1\t_1+\cdots+\gamma_{d-1}\t_{d-1},
    \qquad
    \bgamma=(\gamma_1,\ldots,\gamma_{d-1})
    \in\mathbb{R}^{d-1}.
\]
Since
$\{\t_1,\ldots,\t_{d-1},\n\}$ is an orthonormal basis of
$\mathbb{R}^d$, the change of variables is orthogonal and therefore
\[
    \d\xi=\d\bgamma\,\d s.
\]

\begin{remark}
Writing $\xi=\bgamma\t+s\n$ with $\bgamma\in\R^{d-1}$ and $s\in\mathbb{R}$,
Fubini's theorem gives
\[
    \int_{\n^\perp}
    \int_{\mathbb{R}}
    \left|\widehat f(\bgamma\t+s\n)\right|
    \,\d s\,\d\bgamma
    =
    \|\widehat f\|_{L^1(\mathbb{R}^d)}
    <\infty.
\]
Consequently,
\[
    \widehat f(\bgamma\t+\cdot\,\n)\in L^1(\mathbb{R})
    \qquad\text{for a.e. }\bgamma\in\R^{d-1}.
\]
\end{remark}

Moreover, since $x=R\n$ and $\bgamma \t\cdot \n=0$,
\[
    x\cdot(\bgamma \t+s\n)=Rs.
\]
Consequently, the Fourier representation of $u_\varepsilon$ becomes
\begin{align}
    u_\varepsilon(x)
   & =
    \frac{1}{(2\pi)^d}
    \int_{\mathbb{R}^{d-1}}
    \int_{\mathbb{R}}
    \frac{
        e^{\i Rs}
        \widehat f(\bgamma \t+s\n)
    }{
        (\bgamma \t+s\n)\cdot A(\bgamma \t+s\n)
        -\lambda-i\varepsilon
    }
    \,\d s\,\d\bgamma\nonumber\\
    & =
    \frac{1}{(2\pi)^d}
    \int_{\R^{d-1}}
    \left(
        \int_{L(\bgamma)}
        \frac{e^{ix\cdot\xi}\widehat f(\xi)}
        {\xi\cdot A\xi-\lambda-i\varepsilon}
        \,\d\ell(\xi)
    \right)
    \d\bgamma.
    \label{eq:directional-representation}
\end{align}
where 
\[
L(\bgamma):=\{\bgamma\t+s\n:\,s\in\R\}
\]
is the straight line along $\n$ with fixed $\bgamma\in\R^{d-1}$.

\subsection{Analysis of the characteristic roots}

For fixed $\bgamma\in\mathbb{R}^{d-1}$, we consider the
characteristic roots in the longitudinal variable $s$. The roots are
determined by
\begin{equation}
    (\bgamma \t+s\n)\cdot A(\bgamma \t+s\n)-\lambda-\i\varepsilon=0.
    \label{eq:characteristic-equation}
\end{equation}
For convenience, we introduce
\[
    a:=\n\cdot A\n,
    \qquad
    b(\bgamma):=\bgamma \t\cdot A\n,
    \qquad
    c(\bgamma):=\bgamma \t\cdot A\bgamma \t-\lambda.
\]
Then \eqref{eq:characteristic-equation} becomes
\begin{equation}
    a s^2+2b(\bgamma)s+c(\bgamma)-\i\varepsilon=0.
    \label{eq:quadratic-characteristic}
\end{equation}
Since $A$ is positive definite,
\[
    a=\n\cdot A\n
    =\sum_{j=1}^d a_j\n_j^2
    \geq c_0\sum_{j=1}^d \n_j^2
    =c_0>0.
\]
The characteristic roots are therefore
\begin{equation}\label{eq:characteristic-roots}
    s_\pm(\bgamma,\varepsilon)
    =
    -\frac{b(\bgamma)}{a}
    \pm
    \frac{\sqrt{D(\bgamma)+\i a\varepsilon}}{a},
    \quad\text{  where }
    D(\bgamma)
    :=
    b(\bgamma)^2-a\,c(\bgamma).
\end{equation}
We use the principal branch of the square root, with branch cut
$(-\infty,0]$. Since $a\varepsilon>0$, we have
\[
    \operatorname{Im}
    \sqrt{D(\bgamma)+\i a\varepsilon}>0.
\]
Consequently,
\begin{equation}
    \operatorname{Im}s_+(\bgamma,\varepsilon)>0,
    \qquad
    \operatorname{Im}s_-(\bgamma,\varepsilon)<0.
    \label{eq:root-half-planes}
\end{equation}
For any fixed $\bgamma\in\R^{d-1}$, $s_\pm(\bgamma,\epsilon)$ depend continuously on $\epsilon$, therefore,
\[
    s_\pm(\bgamma,\varepsilon)
    \longrightarrow
    s_\pm(\bgamma)
    :=
    -\frac{b(\bgamma)}{a}
    \pm\frac{\sqrt{D(\bgamma)}}{a},
    \qquad
    \varepsilon\to0^+.
\]

The quantity $D(\bgamma)$ determines the geometry of the real
characteristic roots. Indeed,
\begin{equation}
    \n\cdot A(\bgamma \t+s_\pm(\bgamma,\varepsilon)\n)
    =
    b(\bgamma)+a s_\pm(\bgamma,\varepsilon)
    =
    \pm\sqrt{D(\bgamma)+\i a\varepsilon}.
    \label{eq:normal-component-roots}
\end{equation}
We distinguish three cases.

\paragraph{(i) $D(\bgamma)>0$.}
In this case, $s_+(\bgamma)\neq s_-(\bgamma)$ and
\[
    \n\cdot A(\bgamma \t+s_\pm(\bgamma)\n)
    =
    \pm\sqrt{D(\bgamma)}.
\]
Hence the two limiting points satisfy
\[
\bgamma\t+s_+(\bgamma)\n\in\bF_+,\quad \bgamma\t+s_-(\bgamma)\n\in\bF_-.
\] Thus, the line $L(\bgamma)$
intersects the Fermi surface $\bF$ transversally at two distinct points.

\paragraph{(ii) $D(\bgamma)=0$.}
In this case  $s_+(\bgamma)=s_-(\bgamma):=s_g(\bgamma)$. 
The corresponding point satisfies
\[
    \bgamma\t+s_g(\bgamma)\n\in\bF,
    \qquad
    \n\cdot A(\bgamma\t+s_g(\bgamma)\n)=0.
\]
Thus the line $L(\bgamma)$ is tangent to $\bF$ at
 $\bgamma\t+s_g(\bgamma)\n$, and hence $\bgamma\t+s_g(\bgamma)\n\in\bT_\n$.

\paragraph{(iii) $D(\bgamma)<0$.}
In this case, the limiting characteristic roots $s_\pm(\bgamma)$ are non-real. Consequently, the
line $L(\bgamma)$
does not intersect the Fermi surface $\bF$.

 For a fixed $\delta>0$, we
introduce the decomposition
\begin{align}
    \mathcal R_\delta
    &:=\{\bgamma\in\mathbb R^{d-1}:D(\bgamma)<-\delta\},\\
    \mathcal T_\delta
    &:=\{\bgamma\in\mathbb R^{d-1}:|D(\bgamma)|\leq\delta\},\\
    \mathcal N_\delta
    &:=\{\bgamma\in\mathbb R^{d-1}:D(\bgamma)>\delta\},
\end{align}
so that
\[
    \mathbb R^{d-1}
    =
    \mathcal R_\delta\cup\mathcal T_\delta\cup\mathcal N_\delta.
\]
The region $\mathcal R_\delta$ is uniformly separated from $\bF$,
while $\mathcal N_\delta$ consists of transversal intersections.
The tangency neighbourhood $\mathcal T_\delta$ contains the tangency set
$\mathcal T_0$.

For a small $\eta>0$, we introduce the neighbourhood of the Fermi
surface $\bF$ as
\[
    \mathcal U_\eta(\bF)
    :=
    \left\{
        \xi\in\mathbb R^d:
        |\xi\cdot A\xi-\lambda|<\eta
    \right\}.
\]
See Figure \ref{fig:sample} for the two dimensional example. 
Recall that
\[
    (\bgamma\t+s\n)\cdot A(\bgamma\t+s\n)-\lambda
    =
    a\left(s+\frac{b(\bgamma)}{a}\right)^2
    -\frac{D(\bgamma)}{a}.
\]
Hence, for $\bgamma\in\mathcal R_\delta$,
\[
    (\bgamma\t+s\n)\cdot A(\bgamma\t+s\n)-\lambda
    \geq \frac{\delta}{a}
    \qquad\text{for all }s\in\mathbb R.
\]
Choosing $\eta<\delta/a$, we therefore have
\[
    \mathcal U_\eta(\bF)
    \subset
    \left\{
        \bgamma\t+s\n:
        \bgamma\in\mathcal T_\delta\cup\mathcal N_\delta,\,
        s\in\mathbb R
    \right\}.
\]
Thus, the regular region $\mathcal R_\delta$ does not contribute to the
neighbourhood $\mathcal U_\eta(\bF)$.

We now introduce a cylindrical truncation of the neighbourhood
$\mathcal U_\eta(\bF)$. Let
\[
    \Gamma_\delta
    :=
    \left\{
        \bgamma\in\mathbb R^{d-1}:
        D(\bgamma)=-\frac{\delta}{2}
    \right\},
\]
and define the corresponding cylinder parallel to $\n$ by
\[
    \mathcal C_\delta
    :=
    \left\{
        \bgamma\t+s\n:
        D(\bgamma)>-\frac{\delta}{2},\ s\in\mathbb R
    \right\}.
\]
For $\bgamma\in\Gamma_\delta$, we have
\[
    (\bgamma\t+s\n)\cdot A(\bgamma\t+s\n)-\lambda
    =
    a\left(s+\frac{b(\bgamma)}{a}\right)^2
    +\frac{\delta}{2a}
    \geq \frac{\delta}{2a}>0.
\]
Hence $\Gamma_\delta\cap\bF=\emptyset$, and $\Gamma_\delta$ is
uniformly separated from the Fermi surface. We define the truncated
neighbourhood by
\[
    \widetilde{\mathcal U}_\eta(\bF)
    :=
    \mathcal U_\eta(\bF)\cap\mathcal C_\delta.
\]

For each $\bgamma\in\mathcal T_\delta\cup\mathcal N_\delta$, the
straight line
\[
    L(\bgamma):=\{\bgamma\t+s\n:s\in\mathbb R\}
\]
may intersect $\widetilde{\mathcal U}_\eta(\bF)$ in two disjoint line
segments, one line segment, or not at all. By the construction of the
cylindrical truncation, the lateral boundary of the cylinder $\Gamma_\delta$ has a
positive distance from $\bF$, and hence no nonempty intersection
degenerates to a single point. We parameterize these line segments by
the variable $s$.

\begin{itemize}
\item If $L(\bgamma)\cap\widetilde{\mathcal U}_\eta(\bF)$ consists of
two disjoint line segments, we write
\[
    L(\bgamma)\cap\widetilde{\mathcal U}_\eta(\bF)
    =
    \{\bgamma\t+s\n:s\in I^-(\bgamma)\}
    \cup
    \{\bgamma\t+s\n:s\in I^+(\bgamma)\},
    \qquad
    \operatorname{ind}(\bgamma):=\{-,+\},
\]
where $I^-(\bgamma),I^+(\bgamma)\subset\mathbb R$ are disjoint open
intervals satisfying
\[
    \sup I^-(\bgamma)\leq\inf I^+(\bgamma).
\]

\item If $L(\bgamma)\cap\widetilde{\mathcal U}_\eta(\bF)$ consists of
one line segment, we write
\[
    L(\bgamma)\cap\widetilde{\mathcal U}_\eta(\bF)
    =
    \{\bgamma\t+s\n:s\in I^0(\bgamma)\},
    \qquad
    \operatorname{ind}(\bgamma):=\{0\},
\]
where $I^0(\bgamma)\subset\mathbb R$ is a nonempty open interval.

\item If
$L(\bgamma)\cap\widetilde{\mathcal U}_\eta(\bF)=\emptyset$, we set
\[
    I^0(\bgamma):=\emptyset,
    \qquad
    \operatorname{ind}(\bgamma):=\{0\}.
\]
\end{itemize}

Thus,
\[
    \widetilde{\mathcal U}_\eta(\bF)
    =
    \bigcup_{\bgamma\in\mathcal T_\delta\cup\mathcal N_\delta}
    \bigcup_{i\in\operatorname{ind}(\bgamma)}
    \{\bgamma\t+s\n:s\in I^i(\bgamma)\}.
\]

Let $\mathcal{P}\subset\mathbb{R}^{d-1}$ be a simply connected and
bounded open set such that
\[
    \bigcap_{\bgamma\in\mathcal{P}}
    \operatorname{ind}(\bgamma)
    \neq\emptyset.
\]
For any index
$
    i\in
    \bigcap_{\bgamma\in\mathcal{P}}
    \operatorname{ind}(\bgamma),
$
we define
\[
    \mathcal{U}_{\mathcal{P},i}
    :=
    \left\{
        \bgamma\t+s\n:
        \bgamma\in\mathcal{P},\,
        s\in I^i(\bgamma)
    \right\}
    \subset\mathcal{U}_\eta(\bF).
\]

\subsection{Contour deformation}
\label{sec:cont_deform}

We first state the regularity assumption on the source term that will
be used in the contour deformation.

\begin{assumption}
\label{ass:source_regularity}
Fix $\delta>0$ and choose $\eta$ such that 
  $  0<\eta<\frac{\delta}{2a}. $
Let $\widetilde{\mathcal U}_\eta(\bF)$ be the truncated neighbourhood
of the Fermi surface defined above. There exist finitely many bounded
open sets
\[
    \mathcal P_j\subset\mathcal N_\delta\cup\mathcal T_\delta,
    \qquad j=1,\ldots,J,
\]
such that $\operatorname{ind}(\bgamma)$ is constant on each
$\mathcal P_j$. We denote this common index by $i_j$ and define
\[
    \mathcal U_j
    :=
    \mathcal U_{\mathcal P_j,i_j}\quad\text{ where }\mathcal{U}_j\cap\mathcal{U}_\ell=\emptyset,\,j\neq\ell.
\]
The sets $\mathcal U_j$ cover the truncated neighbourhood in the sense
that
\[
    \overline{\widetilde{\mathcal U}_\eta(\bF)}
    =
    \bigcup_{j=1}^J\overline{\mathcal U_j}.
\]

For each $j$, there exists $\rho_j>0$ such that, for every
$\bgamma\in\mathcal P_j$, the function
\[
    s\longmapsto\widehat f(\bgamma\t+s\n)
\]
admits a holomorphic continuation to a neighbourhood of the rectangle
\[\mathcal{Q}_j(\bgamma):=
    I^{i_j}(\bgamma)+\i[0,\rho_j]=
    \left\{
        z\in\C:
        \operatorname{Re}z\in I^{i_j}(\bgamma),\
        0\leq\operatorname{Im}z\leq\rho_j
    \right\}.\]

No uniform bound with respect to $\bgamma$ is assumed, and no global
holomorphic extension of $\widehat f$ to $\mathbb C^d$ is required.
\end{assumption}

\begin{remark}
It is possible that $\mathcal{P}_j\cap\mathcal{P}_\ell\neq\emptyset$ for $j\neq \ell$. In fact, when $L(\bgamma)\cap\widetilde{\mathcal{U}}_\eta(\bF)$ contains two line segments for a fixed $\bgamma$, it belongs two different sets $\mathcal{P}_j$ and  $\mathcal{P}_\ell$.

Although the transverse sets $P_j$ may overlap, the corresponding
sets $U_j\subset \widetilde U_\eta(F)$ are assumed to be pairwise
disjoint. Thus, the decomposition of $\widetilde U_\eta(F)$ does not
involve any double counting. The overlap of the projections $P_j$
only reflects the possibility that, for a fixed transverse parameter
$\gamma$, the line $L(\gamma)$ intersects the neighborhood in two
distinct components.

\end{remark}

\begin{figure}[h]
    \centering
    \includegraphics[width=0.7\textwidth]{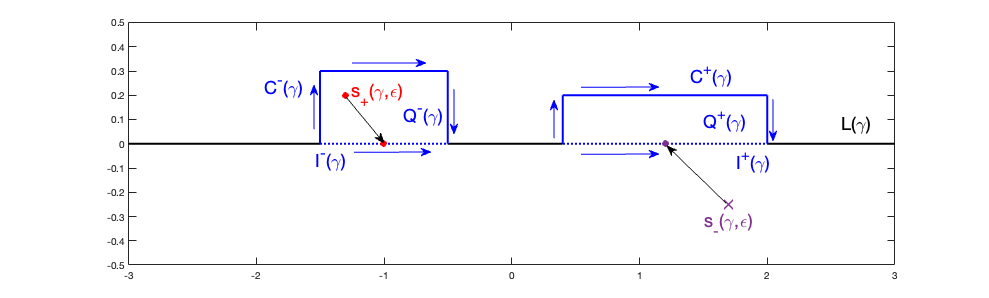}
    \caption{Deformation of $L(\bgamma)$. Red dots: $s_+(\bgamma,\epsilon)$; purple dots: $s_-(\bgamma,\epsilon)$. }
    \label{fig:cd}
\end{figure}

For $\bgamma\in\mathcal{P}_j\subset\mathcal{N}_\delta\cup\mathcal{T}_\delta$ and
$0<\epsilon\leq\epsilon_0$, the characteristic roots
$s_\pm(\bgamma,\epsilon)$ satisfy
\[
    \Im s_+(\bgamma,\epsilon)>0,
    \qquad
    \Im s_-(\bgamma,\epsilon)<0.
\]
By choosing $\epsilon_0>0$ sufficiently small, the roots remain in
the corresponding regions determined by the contour deformation.

We deform the interval $I^{i_j}(\bgamma)$ into the upper half-plane
within $\mathcal Q_j(\bgamma)$, replacing it by the other three sides
of the rectangle. Thus, the deformed contour
$\mathcal C_j(\bgamma)$ consists of the unchanged parts of the real axis
together with
\[
  \mathcal C_j(\bgamma):=
  \partial\mathcal Q_j(\bgamma)
  \setminus I^{i_j}(\bgamma),
\]
with the latter oriented according to the clockwise orientation. For the illustration we refer to Figure \ref{fig:cd}.

By construction, $s_+(\bgamma,\epsilon)$ lies inside
$\mathcal Q_j(\bgamma)$, whereas $s_-(\bgamma,\epsilon)$ remains below
the deformed contour. By continuity of $s_\pm(\bgamma,\epsilon)$ with
respect to $\epsilon$, we may choose $\epsilon_j>0$ sufficiently small
such that, for some $c_j>0$,
\[
    \operatorname{dist}
    \bigl(
        s_\pm(\bgamma,\epsilon),\mathcal C_j(\bgamma)
    \bigr)
    \geq c_j,
    \qquad
    \bgamma\in\mathcal P_j,\quad
    0<\epsilon<\epsilon_j.
\]

\begin{proposition}[Directional contour deformation]
\label{prop:directional_contour}
Let Assumption~\ref{ass:source_regularity} hold. Fix
$j\in\{1,\ldots,J\}$, $\epsilon>0$, and
$\bgamma\in\mathcal P_j$. Let $\mathcal C_j(\bgamma)$ be the
deformed path constructed above. Then
\begin{align}
\nonumber
\int_{ I^{i_j}(\bgamma)}
\frac{
    e^{\i Rs}\widehat f(\bgamma\t+s\n)
}{
    as^2+2b(\bgamma)s+c(\bgamma)-\i\epsilon
}
\,\d s
=&
\pi\i\,
\frac{
    e^{\i R s_+(\bgamma,\epsilon)}
    \widehat f(\bgamma\t+s_+(\bgamma,\epsilon)\n)
}{
    \sqrt{D(\bgamma)+\i a\epsilon}
}\\&
+
\int_{\mathcal C_j(\bgamma)}
\frac{
    e^{\i Rs}\widehat f(\bgamma\t+s\n)
}{
    as^2+2b(\bgamma)s+c(\bgamma)-\i\epsilon
}
\,\d s.\label{eq:contour_deformation}
\end{align}
\end{proposition}

\begin{proof}
For fixed $\epsilon>0$, the characteristic roots
$s_+(\bgamma,\epsilon)$ and $s_-(\bgamma,\epsilon)$ are distinct.
For sufficiently small $\varepsilon>0$, the pole
$s_+(\gamma,\varepsilon)$ lies inside $\mathcal Q_j(\gamma)$, while
$s_-(\gamma,\varepsilon)$ remains outside. The residue theorem therefore gives 
\begin{align*}
\int_{ I^{i_j}(\bgamma)}
\frac{
    e^{\i Rs}\widehat f(\bgamma\t+s\n)
}{
    as^2+2b(\bgamma)s+c(\bgamma)-\i\epsilon
}
\,\d s
=&
2\pi\i\,
\operatorname{Res}\left(
\frac{
    e^{\i Rs}\widehat f(\bgamma\t+s\n)
}{
    as^2+2b(\bgamma)s+c(\bgamma)-\i\epsilon
},
s=s_+(\bgamma,\epsilon)
\right)\\
&+
\int_{ \mathcal{C}_j(\bgamma)}
\frac{
    e^{\i Rs}\widehat f(\bgamma\t+s\n)
}{
    as^2+2b(\bgamma)s+c(\bgamma)-\i\epsilon
}
\,\d s.
\end{align*}
Since
\[
as^2+2b(\bgamma)s+c(\bgamma)-\i\epsilon
=
a(s-s_+(\bgamma,\epsilon))
(s-s_-(\bgamma,\epsilon)),
\]
we have
\[
\begin{aligned}
\operatorname{Res}\left(
\frac{
    e^{\i Rs}\widehat f(\bgamma\t+s\n)
}{
    as^2+2b(\bgamma)s+c(\bgamma)-\i\epsilon
},
s=s_+(\bgamma,\epsilon)\right)
&=
\frac{
    e^{\i R s_+(\bgamma,\epsilon)}
    \widehat f(\bgamma\t+s_+(\bgamma,\epsilon)\n)
}{
    a\bigl(
        s_+(\bgamma,\epsilon)-s_-(\bgamma,\epsilon)
    \bigr)
}
\\
&=
\frac{
    e^{\i R s_+(\bgamma,\epsilon)}
    \widehat f(\bgamma\t+s_+(\bgamma,\epsilon)\n)
}{
    2\sqrt{D(\bgamma)+\i a\epsilon}
}.
\end{aligned}
\]
Substitution into the preceding identity yields
\eqref{eq:contour_deformation}.
\end{proof}

\begin{remark}
The downward deformation is also possible, yielding an analogous
representation involving $s_-(\bgamma,\varepsilon)$. We choose the upward deformation because, for $R>0$,
\[
\left|e^{iRs_+(\bgamma,\varepsilon)}\right|\leq1,
\]
whereas the factor associated with $s_-(\bgamma,\epsilon)$ grows exponentially as
$R\to+\infty$.

\end{remark}

\subsection{Directional representation of the limiting resolvent}
\label{sec:limiting_absorption}

We now establish the limiting absorption principle and derive the
decomposition of the limiting solution into regular, non-tangential, and
tangency contributions.

For $\widehat f\in L^1(\R^d)$, we first decompose
\begin{equation}
\label{eq:split_ueps}
u_\epsilon(x)
=
\frac{1}{(2\pi)^d}
\left(
\int_{\widetilde{\mathcal U}_\eta(\bF)}
+
\int_{\R^d\setminus\widetilde{\mathcal U}_\eta(\bF)}
\right)
\frac{e^{\i x\cdot\xi}\widehat f(\xi)}
{\xi\cdot A\xi-\lambda-\i\epsilon}
\,\d\xi
=:u_\epsilon^{(1)}(x)+u_\epsilon^{(2)}(x).
\end{equation}
On $\R^d\setminus\widetilde{\mathcal U}_\eta(\bF)$, the denominator is
uniformly separated from zero, so $u_\epsilon^{(2)}$ passes directly to
the limit by dominated convergence.

First consider $u_\epsilon^{(1)}$, where the denominator may
vanish. By the finite partition introduced above,
$
    \overline{\widetilde{\mathcal U}_\eta(\bF)}
    =
    \bigcup_{j=1}^J\overline{\mathcal U_j},
$
and therefore
\begin{equation}
\label{eq:split_u1eps}
u_\epsilon^{(1)}(x)
=
\sum_{j=1}^J
\frac{1}{(2\pi)^d}
\int_{\mathcal U_j}
\frac{e^{\i x\cdot\xi}\widehat f(\xi)}
{\xi\cdot A\xi-\lambda-\i\epsilon}
\,\d\xi
=:
\sum_{j=1}^J u_\epsilon^{(1,j)}(x).
\end{equation}
Using the parametrization of each patch $\mathcal U_j$, we obtain
\[
u_\epsilon^{(1,j)}(x)
=
\frac{1}{(2\pi)^d}
\int_{\mathcal P_j}
\int_{I^{i_j}(\bgamma)}
\frac{
    e^{\i Rs}
    \widehat f(\bgamma\t+s\n)
}{
    as^2+2b(\bgamma)s+c(\bgamma)-\i\epsilon
}
\,\d s\,\d\bgamma.
\]

For each fixed $\bgamma\in\mathcal P_j$, the inner integral can be
deformed according to Proposition~\ref{prop:directional_contour}.
The additional integrability required to pass to the limit
$\epsilon\to0^+$ is stated in the following assumption.

\begin{assumption}
\label{asp2}
For each $j$, there exist $\varepsilon_j>0$ and
$g_j\in L^1(\mathcal P_j)$ such that
\[
    \left|
    \frac{
        \widehat f\bigl(
            \bgamma\t+s_+(\bgamma,\varepsilon)\n
        \bigr)
    }{
        \sqrt{D(\bgamma)+\i a\varepsilon}
    }
    \right|
    \leq g_j(\bgamma),
\quad\text{ for }
    0<\varepsilon<\varepsilon_j,
    \qquad
    \bgamma\in\mathcal P_j.
\]
Moreover,
$
    \widehat f\in L^1(\mathcal S),
$
where $\mathcal S$ denotes the deformed integration domain, i.e.
\[
\mathcal{S}:=\R^d\setminus\widetilde{\mathcal{U}}_\eta(\bF)\bigcup\bigcup_{j=1}^J \left\{\bgamma\t+s\n:\,\bgamma\in\mathcal{P}_j,\,s\in \mathcal{C}_j(\bgamma)\right\}.
\]
\end{assumption}

Since $\Im s_+(\bgamma,\varepsilon)>0$ for $\varepsilon>0$ and
$R>0$, we have
\[
\left|e^{\i R s_+(\bgamma,\varepsilon)}\right|\le1.
\]
Thus Assumption~\ref{asp2} provides an $L^1$-majorant for the full
residue integrand. We can now state the limiting absorption principle.

\begin{theorem}[Directional contour representation of the limiting resolvent]
\label{thm:limiting_absorption}
Under Assumptions~\ref{ass:source_regularity} and~\ref{asp2}, for every $x=R\n$,
the limit
\[
    u(x):=\lim_{\epsilon\to0^+}u_\epsilon(x)
\]
exists and is given by
\begin{equation}
\label{eq:limiting_absorption_formula}
u(x)
=
\frac{\pi\i}{(2\pi)^d}
\sum_{j=1}^J
\int_{\mathcal P_j}
\frac{
    e^{\i R s_+(\bgamma)}
    \widehat f(
        \bgamma\t+s_+(\bgamma)\n
    )
}{
    \sqrt{D(\bgamma)}
}
\,\d\bgamma
+
\frac{1}{(2\pi)^d}
\int_{\mathcal S}
\frac{
    e^{\i x\cdot\xi}\widehat f(\xi)
}{
    \xi\cdot A\xi-\lambda
}
\,\d\xi.
\end{equation}
\end{theorem}

\begin{proof}
For each $j$, Proposition~\ref{prop:directional_contour} gives
\begin{align}
u^{(1,j)}_\epsilon(x)
=&
\frac{\pi\i}{(2\pi)^d}
\int_{\mathcal P_j}
\frac{
    e^{\i R s_+(\bgamma,\epsilon)}
    \widehat f(
        \bgamma\t+s_+(\bgamma,\epsilon)\n
    )
}{
    \sqrt{D(\bgamma)+\i a\epsilon}
}
\,\d\bgamma\nonumber\\&+
\frac{1}{(2\pi)^d}\int_{\mathcal{P}_j}
\int_{\mathcal C_j(\bgamma)}
\frac{
    e^{\i x\cdot(\bgamma\t+s\n)}\widehat f(\bgamma\t+s\n)
}{
    (\bgamma\t+s\n)\cdot A(\bgamma\t+s\n)-\lambda-\i\epsilon
}
\,\d s\,\d\bgamma.
\label{eq:contour_integrated}
\end{align}

Summing over $j$ and combining the deformed contours with the
contribution from
$\R^d\setminus\widetilde{\mathcal U}_\eta(\bF)$, we obtain
\begin{equation}
\label{eq:ueps_deformed}
u_\epsilon(x)
=
\frac{\pi\i}{(2\pi)^d}
\sum_{j=1}^J
\int_{\mathcal P_j}
\frac{
    e^{\i R s_+(\bgamma,\epsilon)}
    \widehat f(
        \bgamma\t+s_+(\bgamma,\epsilon)\n
    )
}{
    \sqrt{D(\bgamma)+\i a\epsilon}
}
\,\d\bgamma
+
\frac{1}{(2\pi)^d}
\int_{\mathcal S}
\frac{
    e^{\i x\cdot\xi}\widehat f(\xi)
}{
    \xi\cdot A\xi-\lambda-\i\epsilon
}
\,\d\xi.
\end{equation}

By the definition of $s_+(\bgamma,\epsilon)$,
$
    s_+(\bgamma,\epsilon)
    \longrightarrow
    s_+(\bgamma)$
    as $\epsilon\to0^+.
$
Assumption~\ref{asp2} provides an $L^1$-majorant for the residue
integrand. Hence, by dominated convergence,
\[
\lim_{\epsilon\to0^+}
\int_{\mathcal P_j}
\frac{
    e^{\i R s_+(\bgamma,\epsilon)}
    \widehat f(
        \bgamma\t+s_+(\bgamma,\epsilon)\n
    )
}{
    \sqrt{D(\bgamma)+\i a\epsilon}
}
\,\d\bgamma
=
\int_{\mathcal P_j}
\frac{
    e^{\i R s_+(\bgamma)}
    \widehat f(
        \bgamma\t+s_+(\bgamma)\n
    )
}{
    \sqrt{D(\bgamma)}
}
\,\d\bgamma.
\]

On the deformed contours, the denominator is uniformly bounded away
from zero. Hence, together with $\widehat f\in L^1(\mathcal S)$,
dominated convergence yields
\[
\lim_{\epsilon\to0^+}
\frac{1}{(2\pi)^d}
\int_{\mathcal S}
\frac{
    e^{\i x\cdot\xi}\widehat f(\xi)
}{
    \xi\cdot A\xi-\lambda-\i\epsilon
}
\,\d\xi
=
\frac{1}{(2\pi)^d}
\int_{\mathcal S}
\frac{
    e^{\i x\cdot\xi}\widehat f(\xi)
}{
    \xi\cdot A\xi-\lambda
}
\,\d\xi.
\]
Combining the two limits proves \eqref{eq:limiting_absorption_formula}.
\end{proof}

We next decompose the limiting solution according to the geometry of
the characteristic roots. Fix $\delta>0$ sufficiently small and recall
$
\mathbb R^{d-1}
=
\mathcal R_\delta\cup\mathcal T_\delta\cup\mathcal N_\delta.
$
The deformed-contour contribution is included in the regular part,
while the residue contribution is split according to
$\mathcal N_\delta$ and $\mathcal T_\delta$. Thus
\begin{equation}
\label{eq:three_term_decomposition}
u(x)
=
u^{\mathrm{reg}}(x)
+
u^{\mathrm{nt}}(x)
+
u^{\mathrm{tan}}(x),
\end{equation}
where
\begin{align}
\label{eq:regular}
u^{\mathrm{reg}}(x)
&:=
\frac{1}{(2\pi)^d}
\int_{\mathcal S}
\frac{
    e^{\i Rs}\widehat f(\bgamma\t+s\n)
}{
    as^2+2b(\bgamma)s+c(\bgamma)
}
\,\d s\,\d\bgamma,
\\
\label{eq:non_tangency}
u^{\mathrm{nt}}(x)
&:=
\frac{\i}{(2\pi)^{d-1}}
\sum_{j=1}^J
\int_{\mathcal P_j\cap\mathcal N_\delta}
\frac{
    e^{\i R s_+(\bgamma)}
    \widehat f(
        \bgamma\t+s_+(\bgamma)\n
    )
}{
    2\sqrt{D(\bgamma)}
}
\,\d\bgamma,
\\
\label{eq:tangency}
u^{\mathrm{tan}}(x)
&:=
\frac{\i}{(2\pi)^{d-1}}
\sum_{j=1}^J
\int_{\mathcal P_j\cap\mathcal T_\delta}
\frac{
    e^{\i R s_+(\bgamma)}
    \widehat f(
        \bgamma\t+s_+(\bgamma)\n
    )
}{
    2\sqrt{D(\bgamma)}
}
\,\d\bgamma.
\end{align}

The regular contribution will be treated directly on the deformed
contour, while the non-tangential and tangency contributions will be
analyzed separately below.

\section{Tangency term with $V=0$}
\label{sec:tan_V_0}

The aim of this section is to demonstrate that the interaction between a transverse singularity of the source and the tangency set can produce anomalous far-field decay. We first consider the analytic case and show that the tangency-region contribution cancels to all algebraic orders. We then introduce a model transverse singularity located at the tangency point and derive its algebraic far-field contribution. Finally, by moving the same singularity away from the tangency point, we show that the corresponding tangency-region contribution becomes super-algebraically decaying. This comparison isolates the geometric interaction responsible for the anomalous decay.

\subsection{A two-dimensional setting}
\label{sec:tangency_setting}

We now examine how the interaction between a transverse singularity
of the source and the tangency set affects the far-field decay in the
simplest two-dimensional setting. Let
\[
    d=2,\qquad A=I,\qquad \lambda=k^2,\qquad k>0.
\]
The Fermi surface is then the circle
\[
    \bF=\{\xi\in\R^2:|\xi|=k\}.
\]

We fix the observation direction
\[
    \n=(0,1),
    \qquad
    x=R\n=(0,R),
    \qquad R>0.
\]
The active and tangency parts of the Fermi surface are
\[
    \bF^+
    =
    \{(\xi_1,\xi_2)\in\bF:\xi_2>0\},
\]
and
\[
    \bT_\n
    =
    \{(\xi_1,\xi_2)\in\bF:\xi_2=0\}
    =
    \{(k,0),(-k,0)\}.
\]

Writing
\[
    \t=(1,0),
    \qquad
    \xi=\gamma\t+s\n=(\gamma,s),
\]
we have
\[
    a=1,\qquad
    b(\gamma)=0,\qquad
    c(\gamma)=\gamma^2-k^2,\quad
    D(\gamma)=k^2-\gamma^2.
\]
The characteristic root with positive imaginary part is
\[
    s_+(\gamma,\epsilon)
    =
    \sqrt{k^2-\gamma^2+\i\epsilon}\rightarrow
    s_+(\gamma)
    :=
    s_+(\gamma,0)
    =
    \sqrt{k^2-\gamma^2+\i0},\quad\epsilon\rightarrow 0^+.
\]

For a sufficiently small $\delta>0$, we distinguish the regular,
non-tangential, and tangency parameter regions by
\begin{align*}
\mathcal R
    &=
    (-\infty,-k-\delta)\cup(k+\delta,\infty);\\
    \mathcal N
    &=
    [-k+\delta,k-\delta];\\
     \mathcal T
   & =
    (-k-\delta,-k+\delta)
    \cup
    (k-\delta,k+\delta).
\end{align*}
Here and below, we suppress the dependence on $\delta$ in the
notation.

The tangency-region contribution therefore consists of two local contributions associated with the tangency points \((k,0)\) and \((-k,0)\). We focus on the contribution from \((k,0)\), and choose a cutoff function
$$
\chi_1\in C_0^\infty(\mathbb R),\qquad
\chi_1(r)=1\quad\text{for }|r|\text{ sufficiently small},
$$
in the transverse variable \(\gamma\). This localizes the source to a neighbourhood of one tangency point. The tangency-region contribution will be denoted by $v$ with the definition in terms of \eqref{eq:tangency}.

We will compare sources that are analytic in the transverse variable
near the tangency point with sources having a prescribed transverse
singularity there. In both cases, the source is analytic in the
longitudinal variable $s$ near $(k,0)$.

\subsection{Analytic source decomposition}
\label{sec:ana_source}

We restrict attention to a single patch $\mathcal P$ containing the tangency point $(k,0)$. In the present two-dimensional setting, the transverse variable is $\gamma\in\mathbb R$. We set
$$
    \mathcal P=(k-\delta,k+\delta),
$$
where $\delta>0$ is sufficiently small, and choose $\chi_1\in C_0^\infty(\mathbb R)$ such that
$$
    \operatorname{supp}\chi_1(\gamma-k)\subset\mathcal P
$$
and $\chi_1=1$ in a neighbourhood of $\gamma=k$.

We then write
$$
    \widehat f(\gamma,s)
    =\chi_1(\gamma-k) G(\gamma,s),
$$
where $ G$ is analytic in both variables $(\gamma,s)$ in a neighbourhood of $(k,0)$. Therefore, from \eqref{eq:tangency}, let $v:=u^{\rm tan}$,
$$
v(x)=
\frac{\i}{4\pi}
\int_{k-\delta}^{k+\delta}
\frac{
    e^{\i R\sqrt{k^2-\gamma^2}}\chi_1(\bgamma-k)
    G\left(\gamma,\sqrt{k^2-\gamma^2}\right)
}{
   \sqrt{k^2-\gamma^2}
}
\,\d\gamma.
$$

Splitting the integral at $\gamma=k$ and introducing
\[
s=\sqrt{k^2-\gamma^2}\quad\text{for } \gamma<k,
\qquad
s=\sqrt{\gamma^2-k^2}\quad\text{for } \gamma>k,
\]
we obtain
\begin{equation}
\label{eq:G_int}
v(x)=\frac{\i}{4\pi}v^+(R)+\frac{1}{4\pi}v^-(R),
\end{equation}
where
\begin{equation}
\label{eq:G_int_+}
v^+(R)
=
\int_0^{a_+}
e^{\i Rs}
\frac{
\chi_1\!\left(\sqrt{k^2-s^2}-k\right)
G\!\left(\sqrt{k^2-s^2},s\right)
}{
\sqrt{k^2-s^2}
}
\,\d s,
\end{equation}
and
\begin{equation}
\label{eq:G_int_-}
v^-(R)
=
\int_0^{a_-}
e^{-Rs}
\frac{
\chi_1\!\left(\sqrt{k^2+s^2}-k\right)
G\!\left(\sqrt{k^2+s^2},\i s\right)
}{
\sqrt{k^2+s^2}
}
\,\d s,
\end{equation}
where $a_+=\sqrt{2k\delta-\delta^2}$ and $a_-=\sqrt{2k\delta+\delta^2}$.

We now analyze the far-field behaviour of the two integrals in
\eqref{eq:G_int_+}--\eqref{eq:G_int_-}. The key point is that the endpoint contributions at \(s=0\) cancel term by term at every algebraic order.

Since $\chi(\gamma-k)=1$ in a neighbourhood of $\gamma=k$, the
cutoff does not affect the local behaviour generated by the endpoint
$s=0$. We therefore introduce the local amplitudes
\[
G^+(s)
:=
\frac{G\left(\sqrt{k^2-s^2},s\right)
}{
\sqrt{k^2-s^2}
},\quad
G^-(s)
:=
\frac{
G\left(\sqrt{k^2+s^2},\i s\right)
}{
\sqrt{k^2+s^2}
}.
\]
With a simple computation,
\[
G^-(s)=G^+(\i s).
\]
Both functions are analytic in a neighbourhood of $s=0$. Hence,
\[
G^\pm(s)
=
\sum_{j=0}^\infty a_j^\pm s^j.
\]

\begin{lemma}
\label{lem:tangency_amplitude_coefficients}
With $G^\pm$ defined above, we have
\[
a_j^-=\i^j a_j^+
\qquad
\text{for all }j\geq0.
\]
\end{lemma}

\begin{proof}
Since $G^-(s)=G^+(\i s)$, the Taylor expansion of $G^+$ at $0$ gives
\[
G^-(s)
=G^+(\i s)
=\sum_{j=0}^\infty a_j^+(\i s)^j
=\sum_{j=0}^\infty \i^j a_j^+s^j.
\]
Hence $a_j^-=\i^j a_j^+$ for every $j\geq0$.
\end{proof}

\begin{theorem}[Cancellation of the analytic tangency-region contribution]
\label{thm:analytic_tangency_cancellation}
Let $G$ be analytic in a neighbourhood of $(k,0)$, and let $G^\pm$, $v^\pm$ be defined as above. Then, for every $N\in\mathbb N$,
\[
v(0,R)=\mathcal O(R^{-N})
\qquad
\text{as }R\to\infty.
\]
\end{theorem}

\begin{proof}
Since $\chi_1=1$ in a neighbourhood of 0, the cutoff
does not affect the endpoint expansions. 
Fix $N\in\mathbb N$. Choose $\delta_0>0$ sufficiently small such that
\[
0<\delta_0<a_+<a_-,\quad \chi_1(r)=1\text{ for }|r|\leq\delta_0.
\]
We split the integrals
in \eqref{eq:G_int_+} and \eqref{eq:G_int_-} at $\delta_0$.

The point $\delta_0$ is an artificial splitting point. The part of $v_-$ away from $t=0$ is exponentially small because
of the factor $e^{-Rt}$, while the corresponding part of $v_+$ is
$O(R^{-N})$ for every $N$ by repeated integration by parts, since the
cutoff removes the outer endpoint contribution.  Thus only the  endpoint \(0\) needs to be considered. On $[0,\delta_0]$, Taylor expansion gives
\[
G_+(s)
=
\sum_{j=0}^{N-1}a_j^+s^j+\mathcal O(s^N),
\qquad
G_-(t)
=
\sum_{j=0}^{N-1}a_j^-t^j+\mathcal O(t^N).
\]
The remainder in  $v^+$ is
$\mathcal O(R^{-N-1})$ by repeated integration by parts, while that in $v^-$ has same order after the scaling
$t=R^{-1}y$. Thus it suffices to consider the polynomial terms.

For $j\geq0$, repeated integration by parts gives
\[
\int_0^{\delta_0}e^{\i Rs}s^j\,\d s
=
e^{\i R\delta_0}
\sum_{\ell=0}^{j}
\frac{(-1)^\ell j!}
{(j-\ell)!(\i R)^{\ell+1}}
\delta_0^{j-\ell}
-
\frac{(-1)^j j!}{(\i R)^{j+1}}.
\]
Hence the contribution from the endpoint $s=0$ to the $j$-th term of
$v^+$ is
\[
-\frac{(-1)^j j!}{(\i R)^{j+1}}a_j^+
=
\frac{\i^{j+1}j!}{R^{j+1}}a_j^+.
\]

Similarly,
\[
\int_0^{\delta_0}e^{-Rt}t^j\,\d t
=
e^{-R\delta_0}
\sum_{\ell=0}^{j}
\frac{(-1)^\ell j!}
{(j-\ell)!(-R)^{\ell+1}}
\delta_0^{j-\ell}
-
\frac{(-1)^j j!}{(-R)^{j+1}},
\]
so the contribution from the endpoint $t=0$ to the $j$-th term of
$v^-$ is
\[
\frac{a_j^-j!}{R^{j+1}}.
\]

Using $a_j^-=\i^j a_j^+$ and the prefactors in \eqref{eq:G_int}, the
corresponding contribution to $v(x)$ is
\[
\frac{\i}{4\pi}
\frac{\i^{j+1}j!}{R^{j+1}}a_j^+
+
\frac{1}{4\pi}
\left(
\frac{\i^j j!}{R^{j+1}}a_j^+
\right)
=0.
\]
Thus the endpoint contributions from the tangency point cancel
term by term at every algebraic order. 
Since $N$ is arbitrary, and all remaining contributions are
$\mathcal O(R^{-N})$, we conclude that
\[
v(x)=\mathcal O(R^{-N}),\quad\forall \, N\in\mathbb{N}
\]
\end{proof}

Thus, an analytic source does not produce an algebraically decaying tangency-region contribution. In particular, the tangency-region contribution is super-algebraically small in this case.

\subsection{Non-analytic source decomposition}
\label{sec:non-ana-source}

The preceding result shows that an analytic source produces no algebraic
contribution from the tangency point. We now investigate how this
cancellation is affected by a loss of analyticity in the transverse
variable.

We again work in a neighbourhood of the tangency point and introduce
the two patches
$$
\mathcal P_1=(k-\delta,k),
\qquad
\mathcal P_2=(k,k+\delta),
$$
corresponding to the two sides of the tangency point. To isolate the
effect of the transverse non-analyticity, we consider a source of the
form
$$
\widehat f(\gamma,s)
=\chi_1(\gamma-k) H_\alpha(\gamma,s),
$$
with $\chi_1$ as above and 
$$
H_\alpha(\gamma,s)
=
(\gamma-k)_+^\alpha \chi_2(s),
\qquad
\alpha>-\frac12.
$$
Here $\chi_2\in C_0^\infty(\mathbb R)$ satisfying
$\chi_2=1$ in a neighbourhood of $0$, and
$$
(\gamma-k)_+^\alpha
:=
\begin{cases}
0,&\gamma\leq k,\\
(\gamma-k)^\alpha,&\gamma>k.
\end{cases}
$$
Thus, $\hat{f}(\gamma,s)$ vanishes on $\mathcal P_1$ and equals
$\chi_1(\gamma-k)(\gamma-k)^\alpha$ on $\mathcal P_2$ near the tangency point, with the non-analyticity located at \(\gamma=k\).

\begin{remark}
The source above satisfies Assumption~\ref{asp2}.
Indeed, since $\chi_2=1$ in a neighbourhood of $s=0$, the source is
independent of the longitudinal variable $s$ near the tangency point.
Moreover, since 
\[
D(\gamma)=k^2-\gamma^2,
\]
we have, as $\gamma\to k^+$,
\[
\frac{|H_\alpha(\gamma,s)|}{\sqrt{|D(\gamma)|}}
\asymp
(\gamma-k)^{\alpha-\frac12}.
\]
Finally, since \(\alpha>-\frac12\), the transverse factor \(H_\alpha\) is locally integrable at \(\gamma=k\); together with the compact support of \(\chi_1\) and \(\chi_2\), this yields $\hat{f}\in L^1(S)$.
\end{remark}

Again from \eqref{eq:tangency}, 
\begin{equation}
\label{eq:H_int}
    v(x)
    =
    \frac{\i}{4\pi}v^+(x)
    +
    \frac{1}{4\pi}v^-(x).
\end{equation}
The corresponding expressions for $v^\pm$ are obtained by replacing
$G$ with $H$:
\begin{equation}
\label{eq:H_int_+}
v^+(x)
=
\int_0^{a_+}
e^{\i Rs}
\frac{
    \chi_1(\sqrt{k^2-s^2}-k)
    H_\alpha\left(\sqrt{k^2-s^2},s\right)
}{
    \sqrt{k^2-s^2}
}
\,\d s,
\end{equation}
and
\begin{equation}
\label{eq:H_int_-}
v^-(x)
=
\int_0^{a_-}
e^{-Rs}
\frac{
    \chi_1(\sqrt{k^2+s^2}-k)
    H_\alpha\left(\sqrt{k^2+s^2},s\right)
}{
    \sqrt{k^2+s^2}
}
\,\d s.
\end{equation}

\begin{theorem}[Tangency contribution of a transverse singularity]
\label{thm:nonanalytic_tangency}
Let $H_\alpha$, $v^+$ and $v^-$ be defined as above.  Then
\[
v^+(x)=0,
\]
while
\[
v^-(x)
=
\int_0^{a_-}
e^{-Rt}
\frac{
\chi_1(\sqrt{k^2+t^2}-k)
(\sqrt{k^2+t^2}-k)^\alpha
}{
\sqrt{k^2+t^2}
}
\,\d t.
\]
Moreover, as $R\to\infty$,
\[
v^-(x)
=
\frac{\Gamma(2\alpha+1)}
{k(2k)^\alpha}
R^{-(2\alpha+1)}
+
\mathcal O\bigl(R^{-(2\alpha+3)}\bigr).
\]
In particular, the leading coefficient is nonzero, and hence
\[
v^-(x)\not=o\bigl(R^{-(2\alpha+1)}\bigr).
\]
\end{theorem}

\begin{proof}
The identity $v^+(x)=0$ follows immediately from $\sqrt{k^2-s^2}\leq k$.

For $v^-$, as $t\to0$,
\[
\sqrt{k^2+t^2}-k
=\frac{t^2}{\sqrt{k^2+t^2}+k}=
\frac{t^2}{2k}+\mathcal O(t^4).
\]
Since $\chi_1=1$ near $0$, the integrand therefore has the form
\[
e^{-Rt}
\left(
\frac{1}{k(2k)^\alpha}t^{2\alpha}
+\mathcal O(t^{2\alpha+2})
\right).
\]
Therefore,
\[
v^-(x)
=
\frac{1}{k(2k)^\alpha}
\int_0^{a_-}
e^{-Rt}t^{2\alpha}\,\d t
+
\mathcal O\left(
\int_0^{a_-}
e^{-Rt}t^{2\alpha+2}\,\d t
\right).
\]
From the definition of the Gamma function,
\[
\int_0^\infty e^{-Rt}t^{2\alpha}\,\d t
=
\frac{\Gamma(2\alpha+1)}{R^{2\alpha+1}}.
\]
The replacement of the finite upper limit by $\infty$ produces an
exponentially small error. Hence
\[
v^-(x)
=
\frac{\Gamma(2\alpha+1)}
{k(2k)^\alpha}
R^{-(2\alpha+1)}
+
\mathcal O\bigl(R^{-(2\alpha+3)}\bigr).
\]
This proves the result.

\end{proof}

The range $-\frac12<\alpha<-\frac14$ yields a tangency-region contribution
that decays more slowly than the standard $R^{-1/2}$ far-field term.
We illustrate this anomalous decay with the choice
$
\alpha=-\frac13.
$

\begin{example}
Let $\alpha=-\frac13$. Then
\[
H_{-1/3}\left(\sqrt{k^2+t^2}\right)
=
(2k)^{1/3}t^{-2/3}\bigl(1+O(t^2)\bigr),
\qquad t\to0^+.
\]
Hence Theorem~\ref{thm:nonanalytic_tangency} gives
\[
v^-(R)
=
2^{1/3}k^{-2/3}\Gamma\!\left(\frac13\right)
R^{-1/3}
+O(R^{-7/3}).
\]
Consequently,
\[
v(R)
=
\frac{1}{4\pi}
2^{1/3}k^{-2/3}\Gamma\!\left(\frac13\right)
R^{-1/3}
+O(R^{-7/3}).
\]

Thus the tangency-region contribution decays like $R^{-1/3}$, which is slower
than the usual $R^{-1/2}$ decay of the non-tangential contribution in two
dimensions. In particular, the tangency-region contribution dominates the
standard far-field term.
More generally, the leading order
\[
R^{-(2\alpha+1)}
\]
is slower than $R^{-1/2}$ when
$
-\frac12<\alpha<-\frac14.
$

The corresponding source in physical space is also a genuine function. Indeed, since the Fourier transform is separable, we have
$$
f_{-1/3}(x_1,x_2)=h_{-1/3}(x_1)q(x_2),
$$
where
$$
h_{-1/3}(x_1)
=\frac{e^{ikx_1}}{2\pi}
\int_0^\infty e^{ix_1r}\chi_1(r)r^{-1/3}\,dr,
\qquad
q(x_2)=\frac1{2\pi}\int_{\mathbb R}e^{ix_2s}\chi_2(s)\,ds.
$$
In particular, \(q\in\mathcal S(\mathbb R)\), while the one-sided singularity of the Fourier transform at \(r=0\) gives
$$
h_{-1/3}(x_1)=O(|x_1|^{-2/3})
\qquad\text{as }|x_1|\to\infty.
$$
Consequently,
$$
f_{-1/3}\in L^p(\mathbb R^2)
\qquad\text{for every }p>\frac32,
$$
in particular \(f_{-1/3}\in L^2(\mathbb R^2)\). Thus the anomalous decay exhibited above is produced by a genuine, separable \(L^2\) source, rather than by a distributional or otherwise singular source in physical space.

\end{example}

\subsection{Alignment with the tangency set}

The preceding results indicate that the anomalous tangency-region contribution is tied not merely to transverse non-analyticity, but to its location relative to the tangency set. We now make this dependence explicit.
To this end, consider
\begin{equation*}
\hat{f}_{\alpha,b}(\gamma,s)
:=
\chi_1(\gamma-b)(\gamma-b)_+^\alpha
\chi_2(s),
\qquad
\alpha>-\frac12,
\end{equation*}
and replace $\hat{f}$ in Section \ref{sec:non-ana-source} by $\hat{f}_{\alpha,b}$. Thus, the order of the transverse singularity is unchanged, while its location is shifted from $\gamma=k$ to $\gamma=b$.

Suppose first that
\begin{equation*}
0<b<k.
\end{equation*}
Then $\hat{f}_{\alpha,b}$ is analytic in a neighborhood of the tangency point $(k,0)$. Consequently, from Section \ref{sec:ana_source} the analytic cancellation established above applies to the tangency-region contribution generated by $\hat{f}_{\alpha,b}$, denoted by $v$. In particular,
\begin{equation*}
v
=O(R^{-N})
\qquad\text{for every }N\in\mathbb N.
\end{equation*}
The singularity at $\gamma=b$ may generate a separate non-tangential contribution, but this contribution is not part of the tangency asymptotics considered here.

In contrast, when
\begin{equation*}
b=k,
\end{equation*}
the transverse singularity is located precisely at the tangency point. The analytic cancellation therefore no longer applies. By Theorem~\ref{thm:nonanalytic_tangency},
\begin{equation*}
v
=
C_k R^{-(2\alpha+1)}
+\mathcal{O}\bigl(R^{-(2\alpha+3)}\bigr).
\end{equation*}

Thus, moving the same transverse singularity away from the tangency set changes the tangency-region contribution from an algebraic term to a super-algebraically decaying one. This shows that the anomalously slow decay is not caused by the transverse non-analyticity alone. Rather, it results from the geometric alignment of the transverse singularity with the tangency set.

\section{Realization of the tangency mechanism by a non-trivial potential}
\label{sec:nonzero_potential}

We now show that the tangency-region contribution described above is not
restricted to the free case $V=0$. We construct a potential for which the
first Born term exhibits the tangency contribution described in Section 3.

We use the first Born approximation only as a realization device. The
purpose of this section is not to derive the far-field asymptotics of the
full perturbed resolvent, but to show that the tangency mechanism identified
in the free problem can be generated by a bounded, decaying potential at the
first perturbative level.

We  use the same model as in Section \ref{sec:tan_V_0} and take the point source
$
f=\delta_0
$
and denote by
\[
u_0:=R_0(k^2)\delta_0=\frac{\i}{4}H_0^{(1)}(k|x|)
\]
the corresponding outgoing free-space Green's function. Since $u_0$ is
complex-valued, we first construct a complex-valued potential and then
extract a real-valued one.

At the level of the first-order Born approximation, we have formally
\[
R_V=R_0(I+VR_0)^{-1}
= R_0-R_0VR_0+\cdots.
\]
Therefore,
\[
R_Vf
\approx
R_0f-R_0VR_0f
=
u_0-R_0g,
\]
where
\[
g(x)=V(x)u_0(x).
\]
We now construct a potential for which the first Born correction
$-R_0 V R_0 f$ exhibits a
dominant tangency-region contribution. To this end, we choose a source $g$
satisfying the source assumptions of Section~\ref{sec:tan_V_0}, with the
same transverse non-analyticity as in Section~\ref{sec:non-ana-source}.
Specifically, let
\[
    \widehat g(\xi_1,\xi_2)
    =
    \chi_1(\xi_1-k)(\xi_1-k)_+^\alpha
    \chi_2(\xi_2),
    \qquad
    -\frac12<\alpha<0,
\]
where $\chi_1,\chi_2\in C_0^\infty(\mathbb R)$ are equal to one in
neighbourhoods of $0$. Thus, the tangency-region contribution of $R_0g$ is
exactly of the form analyzed in Theorem~\ref{thm:nonanalytic_tangency}.
We then define
\[
    V_c(x):=\frac{g(x)}{u_0(x)}
    \qquad\Longleftrightarrow\qquad
    V_cu_0=g.
\]
Note that since $H_0^{(1)}(r)$ has no zeros on $\R$, $V_c$ is well defined for all $x\in\R^2\setminus\{0\}$. When $|x|\rightarrow 0$, since $g$ is bounded and $u_0(x)\sim \log|x|$, we simply set $V_c(0)=0$. Therefore, $V_c$ is well defined for all $x\in\R^2$.

\begin{lemma}
\label{lem:potential}
For $-1/2<\alpha<0$, the potential $V_c$ is bounded; moreover, $|V_c(x)|=\mathcal{O}\left(|x|^{-\alpha-\frac{1}{2}}\right)$.
\end{lemma}

\begin{proof}
Since $\widehat g\in L^1(\mathbb R^2)$, the function $g$ is bounded.
Moreover, the separable form of $\widehat g$ gives
\[
    g(x_1,x_2)=h_\alpha(x_1)q(x_2),
\]
where
\[
    h_\alpha(x_1)
    =
    \frac{e^{\i kx_1}}{2\pi}
    \int_0^\infty e^{\i x_1r}\chi_1(r)r^\alpha\,\d r,
    \qquad
    q(x_2)
    =
    \frac{1}{2\pi}\int_{\mathbb R}
    e^{\i x_2s}\chi_2(s)\,\d s.
\]
Here $q\in\mathcal S(\mathbb R)$. The standard endpoint asymptotics for
the one-sided Fourier transform give
\[
    h_\alpha(x_1)
    \lesssim(1+|x_1|)^{-\alpha-1},
    \qquad |x_1|\to\infty.
\] 
Therefore, we have, for every $M>\alpha+1$,
$$
 |g(x_1,x_2)|
 \lesssim
(1+|x_1|)^{-\alpha-1}(1+|x_2|)^{-M}\lesssim (1+|x|)^{-\alpha-1}.
$$
On the other hand, the standard far-field asymptotics of the outgoing
free resolvent give
$$
 |u_0(x)|\asymp |x|^{-1/2},
 \qquad |x|\to\infty.
$$
Thus
$$
 |V_c(x)|=O(|x|^{-\alpha-1/2}),
 \qquad |x|\to\infty.
$$
Near the origin, $g$ is bounded while
$|u_0(x)|\sim |\log|x||$, and therefore $V_c(x)\to0$ as $x\to0$.
In particular, $V_c$ is bounded on $\mathbb R^2$.

\end{proof}

\begin{remark}[Relation between potential decay and tangency decay]
The parameter \(\alpha\) controls both the spatial decay of the constructed potential and the strength of the tangency contribution. More precisely,

$$ |V_c(x)|=O\bigl(|x|^{-\alpha-\frac12}\bigr), \qquad |x|\to\infty, $$
while Theorem 3.4 gives a tangency contribution of order $ R^{-(2\alpha+1)}. $
Thus, increasing \(\alpha\) makes the potential decay faster, but simultaneously makes the tangency contribution decay faster.

It is important to distinguish the tangency mechanism itself from its dominance over the conventional far field. For non-integer \(\alpha\), the transverse factor \((\gamma-k)_+^\alpha\) remains non-analytic at the tangency point, and hence the cancellation mechanism of Theorem 3.2 does not apply. The tangency contribution therefore persists beyond the long-range regime. However, it dominates the standard two-dimensional \(R^{-1/2}\) far-field contribution only when

$$ 2\alpha+1<\frac12, \qquad\text{i.e.}\qquad -\frac12<\alpha<-\frac14. $$

In this regime the constructed potential is long-range. By contrast, for larger \(\alpha\), the potential decays faster while the tangency contribution becomes lower order relative to the conventional far field.

Hence, within this construction, the long-range character of the potential is not essential for the existence of the tangency mechanism itself; rather, it is associated with the possibility that the tangency contribution becomes anomalously slow and dominant.
\end{remark}

\bigskip

We finally pass from the complex-valued potential $V_c$ to a
real-valued one. Write
$$V_c=V_1+\i V_2,
    \qquad
    V_1,V_2\in\mathbb R.
$$
Since $V_c\not\equiv0$, at least one of $V_1$ and $V_2$ is non-trivial.
Moreover,
$$
    V_cu_0
    =
    V_1u_0+\i V_2u_0
    =
    g.
$$
Let $\mathcal C(V)$ denote the coefficient of the leading
$R^{-(2\alpha+1)}$ tangency-region contribution generated by $Vu_0$.
By linearity,
$$
    \mathcal C(V_c)
    =
    \mathcal C(V_1)+\i\mathcal C(V_2).
$$
The construction of $g$ gives
$$
    \mathcal C(V_c)\neq0\quad\Rightarrow\quad|\mathcal C(V_1)|+|\mathcal C(V_2)|\neq 0.
$$
Consequently, at least one of the two real-valued potentials
$V_1$ and $V_2$ produces a non-vanishing tangency-region contribution of order
$R^{-(2\alpha+1)}$ in the first Born term.

In particular, choosing
$$
    \alpha=-\frac13\quad\Rightarrow\quad
    2\alpha+1=\frac13.
$$
Hence at least one real-valued, bounded and decaying potential
produces a non-vanishing tangency-region contribution
$$
    C R^{-1/3},
    \qquad C\neq0,
$$
in the first Born term. Since the standard non-tangential contribution
in dimension two decays like
$
    R^{-1/2},
$
we have $    R^{-1/3}\gg R^{-1/2}$ as $R\to\infty.$

Thus the tangency-region contribution can dominate the conventional
non-tangential far field already at the level of the first Born
approximation.

\begin{remark}
The construction above shows that the tangency mechanism identified in
the free case is not an artefact of the absence of a potential. It can
already be realized by a bounded, decaying real-valued perturbation at
the level of the first Born approximation.

The distinction between mechanism and dominance is relevant here. In
the present family, the long-range character of the potential
corresponds to the parameter regime in which the tangency contribution
can dominate the conventional far field. The relevant quantity is the
effective source $Vu_0$: when its transverse Fourier transform has a
non-analyticity aligned with the tangency set, the resulting tangency
contribution can decay more slowly than the conventional non-tangential
far field. In particular, for $\alpha=-\frac13$, the tangency
contribution is of order $R^{-1/3}$.
\end{remark}

\section*{Declaration of interest}
The author declares that there are no conflicts of interest.

\bibliographystyle{plain}
\bibliography{aa-biblio}

\end{document}